\documentclass[a4paper]{article}

\usepackage[utf8]{inputenc}
\usepackage{lmodern}

\usepackage{amsmath,amssymb,amsthm}
\usepackage{array}
\usepackage{tikz}
\usetikzlibrary{automata}
\tikzstyle{every state}+=[inner sep=0pt]
\tikzstyle{every edge}+=[->]
\tikzstyle{accepting}=[fill=gray!50]
\tikzstyle{extern}=[color=gray!50!black]

\makeatletter
\newenvironment{closedcases}{%
  \matrix@check\cases\env@cases
}{%
  \endarray\right\rbrace%
}
\makeatother

\newcommand\rationals{\ensuremath{\mathbb Q}}
\newcommand\naturals{\ensuremath{\mathbb N}}
\newcommand\integers{\ensuremath{\mathbb Z}}

\newcommand\y{\ensuremath{\mathord\bullet}}
\newcommand\n{\ensuremath{\mathord\circ}}

\newcommand\uu{ -- ++(1,1) node {}}
\newcommand\dd{ -- ++(1,-1) node {}}

\newcommand\ut{ -- ++(1,2) node {}}
\newcommand\dt{ -- ++(1,-2) node {}}

\newtheorem{theorem}{Theorem}
\newtheorem{lemma}[theorem]{Lemma}

\newtheorem{proposition}[theorem]{Proposition}

\newtheorem{defn}[theorem]{Definition}
\newenvironment{definition}{\begin{defn}\textnormal\bgroup}{\egroup\end{defn}}

\let\epsilon\varepsilon
\DeclareMathOperator\inv{inv}
\DeclareMathOperator\cof{Cof}

\newcommand\ldb{\ensuremath{[\![}}
\newcommand\rdb{\ensuremath{]\!]}}

\newcommand\arraynl{\\[\extrarowheight]\hline}
\let\originalendarray\endarray
\renewcommand\endarray{\\[\extrarowheight]\originalendarray}
\newenvironment{mymatrix}{\setlength\extrarowheight{.25em}\begin{array}}
{\end{array}\setlength\extrarowheight{0pt}}

\title{Generalized Dyck paths of bounded height}
\author{Axel Bacher}

\begin{document}

\maketitle

\begin{abstract}
Generalized Dyck paths (or discrete excursions) are one-dimensional paths that
take their steps in a given finite set $S$, start and end at height~$0$, and
remain at a non-negative height. Bousquet-M\'elou showed that the generating
function $E_k$ of excursions of height at most $k$ is of the form
$F_k/F_{k+1}$, where the $F_k$ are polynomials satisfying a linear
recurrence relation. We give a combinatorial interpretation of the polynomials
$F_k$ and of their recurrence relation using a transfer matrix method. We
then extend our method to enumerate discrete meanders (or paths that start at
$0$ and remain at a non-negative height, but may end anywhere). Finally, we
study the particular case where the set $S$ is symmetric and show that several
simplifications occur.
\end{abstract}

\section{Introduction and notations}

A \emph{Dyck path} is a one-dimensional path taking its steps in $\{-1,1\}$,
starting and ending at $0$ and visiting only non-negative points
(Figure~\ref{dyck}, left). It is well-known that the number of Dyck paths of
length $2n$ is the $n$th \emph{Catalan number} $C_n$
\cite[Chapter~6]{stanley2}:
\[C_n = \frac1{n+1}\binom{2n}{n}.\]
Moreover, let $E = E(t)$ be the generating function of Dyck paths. This
generating function satisfies the following algebraic equation:
\[1 - E + t^2E^2 = 0.\]

A \emph{generalized Dyck path} (or \emph{discrete excursion}) takes its steps
in a given finite set $S$ instead of $\{-1,1\}$. A \emph{discrete meander} is
a slightly more general path: it takes its steps in $S$, starts at $0$ and
visits only non-negative points, but may end anywhere (figure~\ref{dyck},
right).

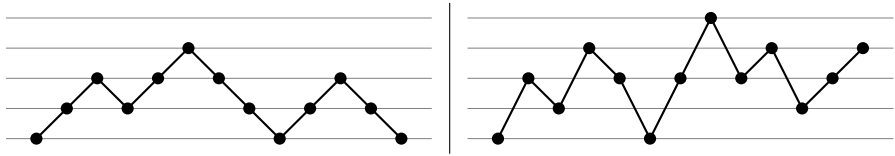
\begin{figure}[htbp]
\hfil%
\begin{tikzpicture}[scale=.4,baseline=0cm]
\tikzstyle{every node}=[circle,fill,inner sep=0pt,minimum size=1.6mm]
\foreach \i in {0,...,4}
    \draw[help lines] (-1,\i) -- ++(14,0);
\draw [thick] (0,0) node {} \uu\uu\dd\uu\uu\dd\dd\dd\uu\uu\dd\dd;
\end{tikzpicture}%
\hfil%
\begin{tikzpicture}[scale=.4,baseline=0cm]
\draw (0,-.5) -- (0,4.5);
\end{tikzpicture}\hfil%
\begin{tikzpicture}[scale=.4,baseline=0cm]
\tikzstyle{every node}=[circle,fill,inner sep=0pt,minimum size=1.6mm]
\foreach \i in {0,...,4}
    \draw[help lines] (-1,\i) -- ++(14,0);
\draw [thick] (0,0) node {} \ut\dd\ut\dd\dt\ut\ut\dt\uu\dt\uu\uu;
\end{tikzpicture}%
\caption{Left: a Dyck path of height~$3$. Right: a discrete meander with set
of steps $S = \{\pm1,\pm2\}$, of height~$4$ and with final height~$3$.}
\label{dyck}
\end{figure}

A large number of papers enumerate generalized Dyck paths with varying amounts
of generality. Methods used include some properties of Laurent series
\cite{gessel}, grammars \cite{labelle,merlini,duchon,ayyer}, or the kernel
method \cite{banderier,bousquet}. Assume that all steps $s$ in $S$ have a
weight $\omega_s$ taken in some field of characteristic~$0$ (typically a field
of fractions with one or several variables over \rationals). Let $E$ and $M$
be the generating functions of excursions and meanders, respectively,
according to the weights~$\omega_s$. Banderier and Flajolet \cite{banderier}
showed that both these generating functions are algebraic. More precisely, let
$a = \max S$ and $b = -\min S$; one may compute polynomials of degree
\smash{$\binom{a+b}{a}$} canceling $E$ and $M$.

In this paper, we consider excursions and meanders with \emph{bounded height},
that is, that never go above a certain level, say $k$. We denote by $E_k$
the generating function of excursions of height at most $k$ and by
$E_{k,\ell}$ the generating function of meanders of height at most $k$ with
final height~$\ell$. We also denote by $M_k$ the generating function of
meanders of height at most $k$ regardless of final height.

Bousquet-Mélou \cite{bousquet} proved that the generating function $E_k$ is of
the form:
\[E_k = \frac{F_k}{F_{k+1}},\]
where the $F_k$ are polynomials in the weights~$\omega_s$. She also proved,
with the use of symmetric functions, that the polynomials~$F_k$ satisfy a
linear recurrence relation of order \smash{$\binom{a+b}{a}$}. In other words,
we have
\[\sum_{k\ge0} F_kz^k = \frac{N(z)}{D(z)},\]
where the polynomial $D(z)$ has degree \smash{$\binom{a+b}{a}$} (the
polynomial $N(z)$ has degree \smash{$\binom{a+b}{a} - a - b$}). Moreover, it
can be seen that the polynomial $D(z)$ cancels the generating function of
excursions $E$.

To our knowledge, the only cases of the generating function $E_{k,\ell}$
that were studied before with general steps are $\ell = 0$, that corresponds
to excursions, and $\ell = k$, that corresponds to \emph{culminating paths}
\cite{ponty}. In this paper, we use a transfer matrix method to compute both
generating functions $E_k$ and $E_{k,\ell}$.

\medskip

We now set some notations. As argued in \cite{bousquet}, excursions of height
at most $k$ are walks in a finite graph, with vertices $\{0,\dotsc,k\}$ and an
arc from $i$ to $j$ if $j-i$ is in $S$. We denote by $A_k$ the adjacency
matrix of this graph. If $s$ is in \integers, let $\beta_s$ be the quantity:
\[\beta_s = \delta_{s,0} - \begin{cases}
\omega_s&\text{if }s\in S\text,\\
0&\text{otherwise.}
\end{cases}\]
In this way, the $(i,j)$ entry of the matrix $1 - A_k$ is $\beta_{j-i}$.

If $m$ and $n$ are integers, we use the notation $\ldb m,n\rdb$ to denote the
set of integers $i$ such that $m\le i\le n$. We also use the notation
$\naturals_{\ge n}$ to denote the set of integers greater than or equal
to~$n$. Finally, if $X$ is a set and $n$ an integer, we call \emph{$n$-subset}
of $X$ a subset of $X$ of cardinality $n$.

\medskip

The paper is organized as follows. Section~\ref{excursions} deals with bounded
excursions, re-proving Bousquet-Mélou's results using a simple transfer matrix
method. This method is expanded in Section~\ref{meanders} to cover bounded
meanders as well. In Section~\ref{symmetric}, we discuss the case where the
set $S$ is symmetric and show that several simplifications occur.


\section{Bounded excursions}\label{excursions}

Our first step to enumerate excursions of height at most $k$ is the same as in
\cite{bousquet}. As these excursions are walks that go from $0$ to $0$ in the
graph described by the matrix $A_k$, the generating function $E_k$
is the $(0,0)$ entry of the matrix $(1 - A_k)^{-1}$ (see
\cite[Chapter~4]{stanley}). Therefore, we have:
\begin{equation}\label{Ek}
E_k = \frac{F_k}{F_{k+1}},
\end{equation}
where $F_k$ is the determinant of the matrix $1 - A_{k-1}$, with the
convention $F_0 = 1$.

As the entry $(i,j)$ of the matrix $1 - A_{k-1}$ is $\beta_{j-i}$, we have, by
the definition of the determinant:
\begin{equation*}
F_k = \sum_{\sigma\in\mathfrak S_k} \epsilon(\sigma)
\prod_{i=0}^{k-1} \beta_{\sigma(i)-i}.
\end{equation*}

In the following, we use this expression to compute $F_k$. The
permutations $\sigma$ that it involves are, of course, bijections from
$\ldb0,k-1\rdb$ to itself; however, we find it more convenient to regard them
as bijections from $\naturals$ to itself that fix all points in
$\naturals_{\ge k}$.

\begin{definition}
Let $I$ be an $a$-subset of the set $\ldb -b,a-1\rdb$. We call
\emph{$I$-permutation} of order $k$ a bijection from
$\naturals$ to $I\cup\naturals_{\ge a}$ that fixes all points in
$\naturals_{\ge k}$.
\end{definition}

Note that for a $I$-permutation to exist when $k < a$, all points in $\ldb
k,a-1\rdb$ must be in the set $I$ since they are fixed points. Also note that
a $\ldb0,a-1\rdb$-permutation of order $k$ is the same as a standard
permutation of order $k$; for this reason, we set $I_0 = \ldb0,a-1\rdb$.

We denote by \smash{$\mathfrak S_k^{(I)}$} the set of $I$-permutations of
order $k$. Let $\sigma$ be a $I$-permutation. We define the number of
\emph{inversions} of $\sigma$, denoted by $\inv(\sigma)$, in the same manner
as a regular permutation:
\[\inv(\sigma) =
\#\bigl\{(i,j)\mid i < j\text{ and }\sigma(i) > \sigma(j)\bigr\}.\]
We define the \emph{signature} of $\sigma$, denoted by $\epsilon(\sigma)$, to
be the number $(-1)^{\inv(\sigma)}$. We also define the quantity
$\beta(\sigma)$ to be:
\[\beta(\sigma) = \epsilon(\sigma)\prod_{i=0}^{k-1}\beta_{\sigma(i)-i}.\]
Finally, we call \emph{head} of $\sigma$ the value $\sigma(0)$; we call
\emph{tail} of $\sigma$ the mapping $\tau$ defined for all $n\in\naturals$ by:
\[\tau(n) = \sigma(n+1) - 1.\]

We now denote by $\mathbf F_k$ the vector indexed by the $a$-subsets of
$\ldb-b,a-1\rdb$ and the entries of which are:
\begin{equation*}
\mathbf F_k[I] = \sum_{\sigma\in\mathfrak S_k^{(I)}} \beta(\sigma).
\end{equation*}
The above remark means that the entry $I_0$ of this vector coincides with
$F_k$.

We compute the vector $\mathbf F_k$ using a simple transfer matrix method, in
a manner similar to \cite[Example~4.7.7 and Proposition~4.7.8a]{stanley}. If
$I$ is a subset of $\ldb-b,a-1\rdb$ and $s$ an integer, we denote by
$\epsilon_s(I)$ the number $-1$ to the power of the number of elements of $I$
lower than~$s$:
\[\epsilon_j(I) = (-1)^{\#\{i\,\in\,I\,\mid\,i\,<\,s\}}.\]
We denote by $\mathbf T$ the matrix whose rows and columns are indexed by the
$a$-subsets of $\ldb-b,a-1\rdb$ and the entries of which are:
\begin{equation}\label{T}
\mathbf T[I,J] = \begin{cases}
\epsilon_s(I)\beta_s&\text{if }I\cup\{a\} = (J+1)\cup\{s\}\text,\\
0&\text{otherwise.}
\end{cases}
\end{equation}

Figure~\ref{fig-T} shows two instances of this matrix. The first corresponds
to the paths with set of steps $S = \{0,\pm1\}$ (Motzkin paths); the second
corresponds to $S = \{0,\pm1,\pm2\}$ (this is known as the \emph{basketball}
problem, and studied in \cite{ayyer}). In both cases, we prefer to represent
the graph $\mathcal G$, the vertices of which are the $a$-subsets of
$\ldb-b,a-1\rdb$ and the adjacency matrix of which is $\mathbf T$.


\begin{figure}[htbp]
\begin{center}
\begin{tikzpicture}[scale=2,>=latex,baseline=0cm]
\tikzstyle{state}+=[inner sep=0pt]
\tikzstyle{every edge}+=[->]
\begin{scope}[yshift=2cm]
\node (10) at (0,0) [state,accepting] {\n\y};
\node (01) at (1,0) [state] {\y\n};
\path (10) edge [loop left] node [left] {$\beta_0$} ();
\path (10) edge [bend left=15] node [above] {$-\beta_1$} (01);
\path (01) edge [bend left=15] node [below] {$\beta_{-1}$} (10);
\end{scope}
\draw (-2.3,1.5) -- ++(4.6,0);
\node (1100) at (0,0) [state,accepting] {\n\n\y\y};
\node (1001) at (-1,-1) [state] {\n\y\y\n};
\node (0110) at (-1,1) [state] {\y\n\n\y};
\node (0011) at (-2,0) [state] {\y\y\n\n};
\node (1010) at (1,0) [state] {\n\y\n\y};
\node (0101) at (2,0) [state] {\y\n\y\n};
\path (1100) edge node [below right] {$\beta_2$} (1001);
\path (1001) edge node [below left] {$\beta_2$} (0011);
\path (0011) edge node [above left] {$\beta_{-2}$} (0110);
\path (0110) edge node [above right] {$\beta_{-2}$} (1100);
\path (1100) edge [loop left] node [left] {$\beta_0$} ();
\path (1001) edge node [left] {$-\beta_0$} (0110);
\path (1100) edge [bend left=15] node [above] {$-\beta_1$} (1010);
\path (1010) edge [bend left=15] node [below] {$\beta_{-1}$} (1100);
\path (1010) edge [bend right=30] node [above right] {$-\beta_1$} (0110);
\path (1001) edge [bend right=30] node [below right] {$\beta_{-1}$} (1010);
\path (1010) edge [bend left=15] node [above] {$\beta_2$} (0101);
\path (0101) edge [bend left=15] node [below] {$\beta_{-2}$} (1010);
\end{tikzpicture}%
\end{center}
\caption{The graph $\mathcal G$ for the sets of steps $S = \{0,\pm1\}$ (above)
and $S = \{0,\pm1,\pm2\}$ (below). In each case, the subset $I$ of
$\ldb-b,a-1\rdb$ corresponding to each vertex is represented by a sequence of
$\y$ and $\n$ (\textit{e.g.} $\n\y\n\y$ corresponds to the subset $\{-1,1\}$
of $\ldb-2,1\rdb$). The vertex corresponding to the set $I_0$ is colored
gray.}
\label{fig-T}
\end{figure}
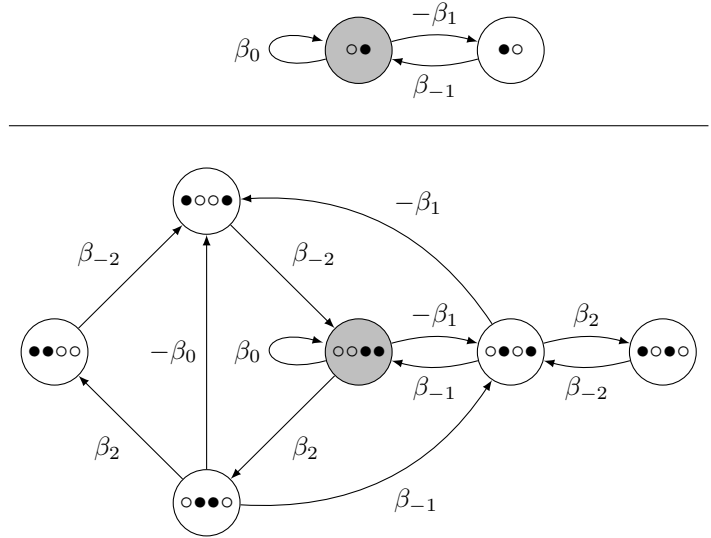

To help the reader familiarise himself with this definition, we start by
stating two basic properties of the matrix~$\mathbf T$. We make use of them
later.

\begin{lemma}\label{lem1}
Let $I$ and $J$ be $a$-subsets of $\ldb-b,a-1\rdb$ and assume that $\mathbf
T[I,J]\ne0$. The two following implications hold.
\begin{enumerate}
\item If $-b$ is in $I$, then $J$ is such that $I\cup\{a\} = (J+1)\cup\{-b\}$.
In this case, we have $\mathbf T[I,J] = \beta_{-b}$.
\item If $a-1$ is not in $J$, then $I$ is equal to $J+1$. In this case, we
have $\mathbf T[I,J] = (-1)^a\beta_a$.
\end{enumerate}
\end{lemma}

We can check this lemma by looking at Figure~\ref{fig-T}: all vertices with a
label starting with a $\y$ have only one outgoing arc, with
label~$\beta_{-b}$; all vertices with a label ending with a $\n$ have only one
ingoing arc, with label~$(-1)^a\beta_a$.

\begin{proof}
By the definition of the matrix $\mathbf T$, if $I$ and $J$ are such that
$\mathbf T[I,J]\ne0$, we have $I\cup\{a\} = (J+1)\cup\{s\}$ for some $s$ in
$\ldb-b,a-1\rdb$.

Let us prove the first implication. As $J$ is a subset of $\ldb-b,a-1\rdb$,
the point $-b$ cannot be in $J+1$. Therefore, if $-b$ is in $I$, we have $s =
-b$. To prove the second implication, we note that the point $a$ is obviously
in $I\cup\{a\}$ and therefore in $(J+1)\cup\{s\}$. Therefore, if $a-1$ is not
in $J$, we have $s = a$.
\end{proof}

We now define some particular vertices of the graph $\mathcal G$, defined as
cyclic permutations of $I_0$. Specifically, if $m$ is such that $-b\le m\le
a$, let $I_m$ be the vertex of $\mathcal G$ defined as the set $\ldb
m,m+a-1\rdb$ modulo $a+b$, where the values modulo $a+b$ are taken in
$\ldb-b,a-1\rdb$. Since $I_{-b} = I_a$, these vertices number $a+b$.

\begin{lemma}\label{lem2}
If $0<m\le a$, any walk in the graph $\mathcal G$ going from $I_m$ visits
$I_0$. If $-b\le m < 0$, any walk in the graph $\mathcal G$ going backwards
from $I_m$ visits $I_0$.
\end{lemma}

\begin{proof}
The first result stems from the fact that $-b$ is in $I_m$ if $0<m\le a$.
Lemma~\ref{lem1} entails that the only arc of $\mathcal G$ going from $I_m$
goes to $I_{m-1}$. Thus, any walk going from $I_m$ reaches $I_0$ in $m$ steps.

Likewise, the second results stems from the fact that $a-1$ is not in $I_m$ if
$-b\le m < 0$. Lemma~\ref{lem1} entails that the only arc of $\mathcal G$
going to $I_m$ goes from $I_{m+1}$. This means that any walk going backwards
from $I_m$ reaches $I_0$.
\end{proof}

\begin{proposition}\label{FkT}
The vectors $\mathbf F_k$ satisfy for $k\ge0$:
\[\mathbf F_{k+1} = \mathbf T\mathbf F_k.\]
\end{proposition}

\begin{proof}
To prove the proposition, we show that for every $a$-subset $I$ of
$\ldb-b,a-1\rdb$, we have:
\[\mathbf F_{k+1}[I] = \sum_J\mathbf T[I,J]\mathbf F_k[J].\]
To do this, we let $\sigma$ be a $I$-permutation of order $k+1$ such that
$\beta(\sigma)\ne0$. Let $s = \sigma(0)$ be the head of $\sigma$ and $\tau$ be
its tail, defined above. Since the number of inversions created by $0$ in the
permutation $\sigma$ is equal to the number of elements of $I$ lower than $s$,
we have:
\[\beta(\sigma) = \epsilon_s(I)\beta_s\beta(\tau).\]

It remains to show that $\tau$ is a $J$-permutation, where
$I\cup\{a\}=(J+1)\cup\{s\}$. Since $\beta_s\ne0$, we have $s\in\ldb-b,a\rdb$.
Thus, we may write:
\[\sigma\colon\naturals\to
\{s\}\cup\bigl(I\cup\{a\}\setminus\{s\}\bigr)\cup\naturals_{\ge a+1}.\]
From this, we deduce:
\[\tau\colon\naturals\to J\cup\naturals_{\ge a}.\]
From its definition, the set $J$ is \textit{a priori} an $a$-subset of
$\ldb-b-1,a-1\rdb$; however, since $\beta(\tau)\ne0$, the point $-b-1$ cannot be
in $J$. This finishes the proof.
\end{proof}

From Proposition~\ref{FkT}, we derive the following theorem, which already
appears in \cite{bousquet} but is found using completely different methods.

\begin{theorem}\label{Fk}
The generating function of the polynomials $F_k$ is:
\[\sum_{k\ge0}F_kz^k = \frac{N(z)}{D(z)}\text,\]
where $D(z)$ is the determinant of $1 - z\mathbf T$ and $N(z)$ is the
$(I_0,I_0)$ cofactor of the same matrix.

Moreover, the degree of the polynomial $D(z)$ is \smash{$\binom{a+b}{a}$}; the
degree of $N(z)$ is \smash{$\binom{a+b}{a} - a - b$}.
\end{theorem}

\begin{proof}
The only possible $I$-permutation of order $0$ is the identity, which is an
$I_0$-permutation. This implies that the polynomial \smash{$F_0^{(I)}$} is $1$
if $I = I_0$ and $0$ otherwise. With Proposition~\ref{FkT}, we deduce that the
polynomial $F_k$ is equal to the entry $(I_0,I_0)$ in the matrix $\mathbf
T^k$.

The generating function \smash{$\sum_{k\ge0}F_kz^k$} is therefore equal
to the entry $(I_0,I_0)$ in the matrix $(1 - z\mathbf T)^{-1}$. The announced
form follows from Cramer's rule.

To compute the degree of the polynomial $D(z)$, we let \smash{$d =
\binom{a+b}{a}$}. Since $\mathbf T$ is a $d\times d$ matrix, the polynomial
$D(z)$ has degree at most $d$; the coefficient of $z^d$ in this polynomial is
(up to a sign) $\det(\mathbf T)$. Denoting by $\mathfrak S_d$ the set of
permutations of the set of $a$-subsets of $\ldb-b,a-1\rdb$, we have:
\[\det(\mathbf T) = \sum_{\pi\in\mathfrak S_d}\epsilon(\pi)\prod_I\mathbf
T[I,\pi(I)].\]
Let $\pi$ be a permutation with a nonzero contribution in this sum.
Lemma~\ref{lem1} asserts that:
\begin{enumerate}
\item if $-b$ is in $I$, then $I\cup\{a\} = \bigl(\pi(I)+1\bigr)\cup\{-b\}$;
\item if $a-1$ is not in $\pi(I)$, then $I = \pi(I)+1$.
\end{enumerate}
Condition~1 determines $\pi(I)$ if $-b$ is in~$I$; Condition~2 determines
$\pi(I)$ if $-b$ is not in~$I$. The permutation $\pi$ is thus uniquely
determined. Moreover, the values of $\mathbf T[I,\pi(I)]$ are given by
Lemma~\ref{lem1}. This gives, up to a sign, the value of $\det(\mathbf T)$:
\[\det(\mathbf T) =
\pm\bigl.\beta_{-b}\bigr.^{\textstyle\binom{a+b-1}{a-1}}
\bigl.\beta_a\bigr.^{\textstyle\binom{a+b-1}{a}},\]
which is nonzero. Therefore, the polynomial $D(z)$ has degree $d$.

To compute the degree of $N(z)$, we adopt a more combinatorial point of view:
we regard the determinant $D(z)$ as the generating function of
\emph{configurations of cycles} of the graph~$\mathcal G$, counted up to a
sign, where $z$ accounts for the number of vertices visited by the
configuration. Let $\pi_0$ be the unique permutation, defined above, that
contributes to the dominant term of $D(z)$; the permutation $\pi_0$ can be
interpreted as the only configuration of cycles that visits all vertices
of~$\mathcal G$.

The cofactor $N(z)$ is the generating function of the configurations of cycles
that avoid the vertex~$I_0$. By Lemma~\ref{lem2}, such a configuration cannot
visit any of the vertices $I_m$ (since if it would, it would also visit
$I_0$). This means that $N(z)$ has degree at most $d - a - b$. Moreover, we
easily check that the vertices $I_m$ form a cycle of the configuration
$\pi_0$; by removing this cycle, we thus obtain a configuration of cycles
visiting $d - a - b$ vertices and not visiting $I_0$. Therefore, the degree of
$N(z)$ is exactly $d - a - b$.
\end{proof}

As examples, let us compute the polynomials $D(z)$ and $N(z)$ corresponding to
the two examples of Figure~\ref{fig-T}. In the case of Motzkin paths $S =
\{0,\pm1\}$, let us take the weights $\omega_0 = 0$ and $\omega_1 =
\omega_{-1} = t$ (\textit{i.e.}, the case of standard Dyck paths).
Theorem~\ref{Fk} shows that the polynomials $F_k$ satisfy:
\[\sum_{k\ge0}F_kz^k = \frac1{1 - z + t^2z^2},\]
which is equivalent the recurrence relation:
\begin{align*}
F_0 &= F_1 = 1,\\
F_k &= F_{k-1} - t^2F_{k-2}&\text{ if }k\ge2.
\end{align*}
These polynomials are commonly known as the \emph{Fibonacci polynomials}, due
to the similarities with the recurrence relation of the Fibonacci numbers.

In the basketball case $S = \{0,\pm1,\pm2\}$, let us take the weights
$\omega_0 = 0$, $\omega_1 = \omega_{-1} = t_1$ and $\omega_2 = \omega_{-2} =
t_2$. In this case, the polynomials $D(z)$ and $N(z)$ are:
\begin{align*}
D(z) &= (1 + t_2z)^2(1 - z - 2t_2z + t_1^2 z^2 + 2t_2z^2 + 2t_2^2z^2 -
t_2^2z^3 - 2t_2^3z^3 + t_2^4z^4);\\
N(z) &= (1 + t_2z)(1 - t_2z).
\end{align*}
As we can see, a simplification by a factor of $1 + t_2z$ occurs in the
computation of the fraction $N(z)/D(z)$. This implies that the polynomials
$F_k$ follow a linear recurrence relation of order~$5$ instead of~$6$, as
shown in \cite{ayyer}. The factorisation of the polynomial $D(z)$ is also
predicted in \cite{bousquet}; it is linked to the fact that the set $S$ and
the weights $\beta_s$ are \emph{symmetric}.

\section{Bounded meanders}\label{meanders}

We enumerate bounded meanders in the same manner as bounded excursions.
Specifically, the generating function \smash{$E_{k,\ell}$} counts walks from
$0$ to $\ell$ in the graph described by the matrix $A_k$. Therefore, we
have
\begin{equation}\label{Ekl}
E_{k,\ell} = \frac{F_{k,\ell}}{F_{k+1}},
\end{equation}
where $F_{k,\ell}$ is the $(\ell,0)$ cofactor of the matrix $1 - A_k$.
Obviously, we have $F_{k,0} = F_k$.

We again compute the cofactors $F_{k,\ell}$ using a transfer matrix
method. We start by writing $F_{k,\ell}$ in terms of the determinant of the
matrix $1 - A_k$ with the $\ell$th row and the $0$th column cut. This
reads:
\begin{align*}
F_{k,\ell} &= (-1)^\ell\det\Biggl(\begin{closedcases}
\beta_{j-i+1}&\text{if $i < \ell$}\\
\beta_{j-i}&\text{if $i\ge\ell$}
\end{closedcases}\Biggr)_{0\le i,j\le k-1}\\
&= (-1)^\ell\sum_{\sigma\in\mathfrak S_k} \epsilon(\sigma)
\prod_{i=0}^{\ell-1}\beta_{\sigma(i)-i+1}
\prod_{i=\ell}^{k-1}\beta_{\sigma(i)-i}.
\end{align*}
For every integer~$s$, we set \smash{$\widetilde\beta_s = -\beta_{s+1}$}. We
also define, for every permutation $\sigma$ of order~$k$, the quantity:
\[\beta_\ell(\sigma) = \epsilon(\sigma)
\prod_{i=0}^{\ell-1}\widetilde\beta_{i-\sigma(i)}
\prod_{i=\ell}^{k-1}\beta_{i-\sigma(i)}.\]
We rewrite the above formula as:
\[F_{k,\ell} = \sum_{\sigma\in\mathfrak S_k} \beta_\ell(\sigma).\]
Let now $I$ be an $a-1$-subset of the set $\ldb-b-1,a-2\rdb$. Let
\smash{$\widetilde{\mathfrak S}_k^{(I)}$} be the set of $I$-permutations of
order~$k$ with respect to the set \smash{$\widetilde S = S-1$}. We define the
following polynomial:
\begin{equation*}
F_{k,\ell}^{(I)} =
\sum_{\sigma\in\widetilde{\mathfrak S}_k^{(I)}} \beta_\ell(\sigma).
\end{equation*}
We also denote by $\mathbf F_{k,\ell}$ the vector whose $I$-component is
\smash{$F_{k,\ell}^{(I)}$}.


Let \smash{$\widetilde{\mathbf T}$} be the transfer matrix defined in
Section~\ref{excursions} corresponding to the set \smash{$\widetilde S$} and
the weights \smash{$\widetilde\beta_s$}. Let $\mathbf U$ be the matrix whose
rows are indexed by the $a-1$-subsets of $\ldb-b-1,a-2\rdb$, whose columns are
indexed by the $a$-subsets of $\ldb-b,a-1\rdb$ and whose entries are:
\begin{equation}\label{U}
\mathbf U[I,J] =
\begin{cases}
1&\text{if }I\cup\{a-1\} = J,\\
0&\text{otherwise.}
\end{cases}
\end{equation}

We also let \smash{$\widetilde{\mathcal G}$} be the graph with adjacency
matrix \smash{$\widetilde{\mathbf T}$}. Let $\mathcal H$ be the graph
\smash{$\mathcal G\cup\widetilde{\mathcal G}$}, with additional arcs between
the vertices of $\mathcal G$ and \smash{$\widetilde{\mathcal G}$} coded by the
matrix $\mathbf U$. Examples are shown in Figure~\ref{fig-U}.

\begin{proposition}\label{FklT}
The vectors $\mathbf F_{k,\ell}$ satisfy for
$k\ge\ell\ge0$:
\begin{align*}
\mathbf F_{k+1,\ell+1} &= \widetilde{\mathbf T}\mathbf F_{k,\ell}\text;\\
\mathbf F_{k,0} &= \mathbf U\mathbf F_k\text.
\end{align*}
\end{proposition}

\begin{proof}
The proof of the first identity follows the same lines as that of
Proposition~\ref{FkT}. Let $I$ be an $a-1$-subset of $\ldb-b-1,a-2\rdb$ and
$\sigma$ be in \smash{$\widetilde{\mathfrak S}_{k+1}^{(I)}$}. Let $s$ be the
head and $\tau$ be the tail of $\sigma$. We have:
\[\beta_{\ell+1}(\sigma)=\epsilon_s(I)\widetilde\beta_s\beta_\ell(\tau).\]
The rest of the proof is identical to that of Proposition~\ref{FkT}.

To prove the second identity, we let $I$ be an $a-1$-subset of
$\ldb-b-1,a-2\rdb$ and $\sigma$ be in \smash{$\widetilde{\mathfrak
S}_k^{(I)}$}. By definition, we have:
\[\beta_0(\sigma) = \beta(\sigma).\]
Moreover, assume that $\beta_0(\sigma) = \beta(\sigma) \ne0$. This means that
$-b-1$ cannot be in $I$. Therefore, the set $J = I\cup\{a-1\}$ is an
$a$-subset of $\ldb-b,a-1\rdb$. Again by definition, the mapping $\sigma$ is
in \smash{$\mathfrak S_k^{(J)}$}. This completes the proof.
\end{proof}

\begin{theorem}\label{Fkl}
The bivariate generating function of the polynomials $F_{k,\ell}$ satisfies:
\begin{equation*}
\sum_{k\ge\ell\ge 0} F_{k,\ell}u^\ell z^k = \sum_{k\ge0}F_k(u)z^k =
\frac{\widetilde N(u,z)}{\widetilde D(uz)D(z)}\text,
\end{equation*}
where $D(z)$ is the determinant of $1 - z\mathbf T$, \smash{$\widetilde
D(z)$} is the determinant of \smash{$1 - z\widetilde{\mathbf T}$}, and
\smash{$\widetilde N(u,z)$} may be computed as:
\[\widetilde N(u,z) = \sum_{\mathbf U[I,J] = 1}
\cof[I,\widetilde I_0]\bigl(1 - uz\widetilde{\mathbf T}\bigr)
\cof[I_0,J]\bigl(1 - z\mathbf T\bigr).\]

The polynomial $D(z)$ has degree \smash{$\binom{a+b}{a}$}, the polynomial
\smash{$\widetilde D(z)$} has degree \smash{$\binom{a+b}{a-1}$}, and the
polynomial \smash{$\widetilde N(u,z)$} has a dominant term in
\[\bigl.z\bigr.^{\textstyle\binom{a+b+1}{a} - a - b - 1}
  \bigl.u\bigr.^{\textstyle\binom{a+b}{a-1} - a}.\]
\end{theorem}

\begin{proof}
As the only $I_0$-permutation of order $0$ is the identity,
Proposition~\ref{FklT} entails that the polynomial $F_{k,\ell}$ is equal to
the entry \smash{$(\widetilde I_0,I_0)$} in the matrix
\smash{$\widetilde{\mathbf T}^\ell\mathbf U\mathbf T^{k-\ell}$}. In other
terms, $F_{k,\ell}$ is the generating function of walks from
\smash{$\widetilde I_0$} to $I_0$ in the graph $\mathcal H$ taking $\ell$
steps in the graph \smash{$\mathcal G$} and $k-\ell$ steps in the graph
$\mathcal G$. The generating function \smash{$\sum_{k,\ell}F_{k,\ell}u^\ell
z^k$} is thus equal to the entry \smash{$(\widetilde I_0,I_0)$} in the matrix
\smash{$(1 - wz\widetilde{\mathbf T})^{-1}\mathbf U(1 - z\mathbf T)^{-1}$}.
This yields the announced form.

Let us now compute the degrees; let \smash{$\tilde d = \binom{a+b}{a-1}$}.
Theorem~\ref{Fk} shows that the polynomial \smash{$\widetilde D(z)$} has
degree \smash{$\tilde d$}.

To compute the dominant term of the polynomial $N(u,z)$, we first remark that
if $I$ and $J$ are such that $\mathbf U[I,J] = 1$, we have $-b-1\not\in I$ and
$a-1\in J$. This implies that Lemma~\ref{lem1} is still valid in the graph
$\mathcal H$. Let \smash{$\widetilde I_m$}, for $-b-1\le m\le a-1$, be the
vertices of \smash{$\widetilde{\mathcal G}$} defined in the same way as the
vertices $I_m$. Since Lemma~\ref{lem1} is valid, Lemma~\ref{lem2} is also
valid regarding both the vertices $I_m$ and \smash{$\widetilde I_m$}.

We now consider the polynomial $N(u,z)$. This polynomial is the generating
function of configurations in the graph $\mathcal H$ composed of elementary
cycles and a self-avoiding walk going from \smash{$\widetilde I_0$} to $I_0$;
the variable $u$ takes into account the number of arcs of
\smash{$\widetilde{\mathcal G}$} in the configuration. By Lemma~\ref{lem2},
such a configuration cannot visit a vertex \smash{$\widetilde I_m$} with $0\le
m\le a-1$ (since it would contain an arc going into \smash{$\widetilde I_0$}),
nor can it visit a vertex $I_m$ with $-b\le m<0$ (since it would contain an
arc going from $I_0$). This proves that the dominant term is at most in
\[z^{\tilde d-a+d-b-1}u^{\tilde d-a}.\]

Let \smash{$\widetilde\pi_0$} and $\pi_0$ be the only configurations of cycles
visiting all vertices of the graphs \smash{$\widetilde{\mathcal G}$} and
$\mathcal G$, respectively. Consider the configuration consisting of:
\begin{itemize}
\item all cycles of \smash{$\widetilde\pi_0$} except the one containing the
vertices \smash{$\widetilde I_m$};
\item all cycles of $\pi_0$ except the one containing the vertices $I_m$;
\item the self-avoiding walk
\smash{$\widetilde I_0\to\dotsb\to\widetilde I_{-b}\to I_a\to\dotsb\to I_0$}.
\end{itemize}
We check that this configuration contains \smash{$\tilde d-a$} arcs of
\smash{$\mathcal G$} and $d - b - 1$ arcs of $\mathcal G$. We thus derive the
dominant term of the polynomial $N(u,z)$.
\end{proof}

We now consider the generating function $M_k$ of all meanders of height at
most $k$ regardless of final height. From \eqref{Ekl}, we find
\begin{equation}\label{Mk}
M_k = \frac{F_k(1)}{F_{k+1}},
\end{equation}
where the polynomial $G_k$ is the sum of $F_{k,\ell}$ for all $\ell$
between $0$ and $k$. The generating function of the polynomials $G_k$ is
found by setting $w = 1$ in the expression of Theorem~\ref{Fkl}; this proves
that the polynomials $G_k$ follow a linear recurrence relation of order
\smash{$\binom{a+b}{a} + \binom{a+b}{a-1} = \binom{a+b+1}{a}$}.

Let us now take the two examples detailed in Section~\ref{excursions}. The
case of Dyck paths (Figure~\ref{fig-U}, left) is very simple since the graph
\smash{$\widetilde{\mathcal G}$} has only one vertex. If we set $\omega_0 = 0$
and $\omega_1 = \omega_{-1} = t$, the generating function of the polynomials
$F_{k,\ell}$ is:
\[\sum_{k\ge\ell\ge0}F_{k,\ell}u^\ell z^k = \frac{1}{(1-tuz)(1-z+t^2z^2)}.\]
In other words, we have:
\[F_{k,\ell} = t^\ell F_{k-\ell}.\]

Let us now examine the case where $S = \{0,\pm1,\pm2\}$, $\omega_0 = 0$,
$\omega_1 = \omega_{-1} = t_1$ and $\omega_2 = \omega_{-2} = t_2$
(Figure~\ref{fig-U}, right). The polynomials \smash{$\widetilde D(z)$} and
\smash{$\widetilde N(u,z)$} are:
\begin{align*}
\widetilde D(z) &= 1 - t_1z - t_2z^2 - t_1t_2^2z^3 + t_2^4z^4;\\
\widetilde N(u,z) &=
(1 + t_2z)(1 - t_2z + t_1t_2uz^2 - t_2^3u^2z^3 + t_2^4u^2z^4).
\end{align*}
Once again, a simplification by a factor of $1 + t_2z$ occurs in the
computation of the generating function
\smash{$\sum_{k,\ell}F_{k,\ell}u^\ell z^k$}.

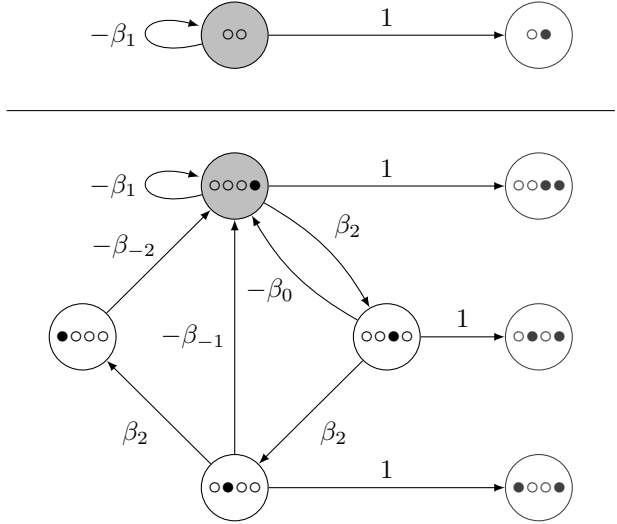
\begin{figure}[htbp]
\hfil%
\begin{tikzpicture}[scale=2,>=latex,baseline=0cm]
\begin{scope}[yshift=1cm]
\node (11) at (0,0) [state,accepting] {\n\n};
\path (11) edge [loop left] node [left] {$-\beta_1$} ();
\node (10) at (2,0) [state,extern] {\n\y};
\path (11) edge node [above] {$1$} (10);
\end{scope}
\draw (-1.5,.5) -- (2.5,.5);
\node (1110) at (0,0) [state,accepting] {\n\n\n\y};
\node (1101) at (1,-1) [state] {\n\n\y\n};
\node (1011) at (0,-2) [state] {\n\y\n\n};
\node (0111) at (-1,-1) [state] {\y\n\n\n};
\path (1110) edge [bend left=15] node [above right] {$\beta_2$} (1101);
\path (1101) edge node [below right] {$\beta_2$} (1011);
\path (1011) edge node [below left] {$\beta_2$} (0111);
\path (0111) edge node [above left=-.25em] {$-\beta_{-2}$} (1110);
\path (1011) edge node [left] {$-\beta_{-1}$} (1110);
\path (1101) edge [bend left=15] node [below left=-.25em] {$-\beta_0$} (1110);
\path (1110) edge [loop left] node [left] {$-\beta_1$} ();
\node (1100) at (2,0) [state,extern] {\n\n\y\y};
\node (1010) at (2,-1) [state,extern] {\n\y\n\y};
\node (0110) at (2,-2) [state,extern] {\y\n\n\y};
\path (1110) edge node [above] {$1$} (1100);
\path (1101) edge node [above] {$1$} (1010);
\path (1011) edge node [above] {$1$} (0110);
\end{tikzpicture}%
\caption{The graphs \smash{$\widetilde{\mathcal G}$} for the sets of steps $S
= \{0,\pm1\}$ (above) and $S = \{0,\pm1,\pm2\}$ (below). Graphical conventions
are identical to those of Figure~\ref{fig-T}, with the vertex corresponding to
\smash{$\widetilde I_0$} colored gray. Arcs with weight~$1$, coded by the
matrix~$\mathbf U$ and leading to vertices of the graph $\mathcal G$ shown in
Figure~\ref{fig-T} are also shown.}
\label{fig-U}
\end{figure}

\section{Symmetric set of steps}\label{symmetric}

We now consider the special case where the set of steps $S$ is
\emph{symmetric}, that is, where $-S = S$ and $\omega_{-i} = \omega_i$ for all
$i\in S$. Bousquet-Mélou already considers this case and shows \cite{bousquet}
that the generating function $E(t)$ is canceled by a polynomial of degree
$2^a$ instead of \smash{$\binom{2a}{a}$}.

While we were not able to recover Bousquet-Mélou's results with our methods,
we show that another phenomenon occurs: namely, the polynomials $F_k$ factor
into two parts, and simplifications occur in the computation of the generating
functions of meanders $E_{k,\ell}$ and $M_k$. These simplifications are
basically due to the fact that as $S$ is symmetric, the entry $(i,j)$ of the
matrix $A_k$ is identical to the entry $(k-i,k-j)$. Define the following two
matrices:
\begin{align*}
A_k^+ &= \Biggl(\omega_{j-i} +
\begin{closedcases}
\omega_{k-j-i}&\text{if $j < k/2$}\\
0&\text{if $j = k/2$}
\end{closedcases}
\Biggr)_{0\le i,j \le k/2}\text,\\
A_k^- &= \Bigl(\omega_{j-i}-\omega_{k-j-i}\Bigr)_{0\le i,j < k/2}
\end{align*}
(all values of $\omega_s$ are taken to be $0$ if $s$ is not in $S$). If $k$ is
an odd number, both matrices have the same dimension and the condition $j =
k/2$ never occurs; if $k$ is an even number, the dimension of
\smash{$A^+_{k+1}$} is one more than that of \smash{$A^-_{k+1}$}. In both
cases, the sum of the two dimensions is $k+1$. The graphs with adjacency
matrices \smash{$A_k^+$} and \smash{$A_k^-$} are illustrated in
Figure~\ref{fig-quotients}.

\begin{figure}[htbp]
\begin{center}
\begin{tikzpicture}[scale=1.6,>=latex,baseline=0cm]
\tikzstyle{every state}+=[minimum size=2em]
\begin{scope}
\foreach \i in {0,...,5}
    \node (\i) at (\i,0) [state] {$\i$};
\path (0) edge [bend left=15] node [above] {$t$} (1);
\path (1) edge [bend left=15] node [above] {$t$} (2);
\path (2) edge [bend left=15] node [above] {$t$} (3);
\path (3) edge [bend left=15] node [above] {$t$} (4);
\path (4) edge [bend left=15] node [above] {$t$} (5);
\path (1) edge [bend left=15] node [below] {$t$} (0);
\path (2) edge [bend left=15] node [below] {$t$} (1);
\path (3) edge [bend left=15] node [below] {$t$} (2);
\path (4) edge [bend left=15] node [below] {$t$} (3);
\path (5) edge [bend left=15] node [below] {$t$} (4);
\end{scope}
\end{tikzpicture}\\
\begin{tikzpicture}[scale=1.6,>=latex,baseline=0cm]
\tikzstyle{every state}+=[minimum size=2em]
\begin{scope}[yshift=-1cm, xshift=-.5cm]
\foreach \i in {0,...,2}
    \node (\i) at (\i,0) [state] {$\i$};
\path (0) edge [bend left=15] node [above] {$t$} (1);
\path (1) edge [bend left=15] node [above] {$t$} (2);
\path (1) edge [bend left=15] node [below] {$t$} (0);
\path (2) edge [bend left=15] node [below] {$t$} (1);
\path (2) edge [loop right] node [right] {$t$} ();
\end{scope}
\begin{scope}[yshift=-1cm,xshift=3.5cm]
\foreach \i in {0,...,2}
    \node (\i) at (\i,0) [state] {$\i$};
\path (0) edge [bend left=15] node [above] {$t$} (1);
\path (1) edge [bend left=15] node [above] {$t$} (2);
\path (1) edge [bend left=15] node [below] {$t$} (0);
\path (2) edge [bend left=15] node [below] {$t$} (1);
\path (2) edge [loop right] node [right] {$-t$} ();
\end{scope}
\end{tikzpicture}%
\end{center}
\caption{The graphs corresponding to the three matrices $A_3$,
\smash{$A_3^+$} and \smash{$A_3^-$} with $S = \{\pm1\}$ and $\omega_1
= \omega_{-1} = t$. They differ only by the vertex~$3$.}
\label{fig-quotients}
\end{figure}
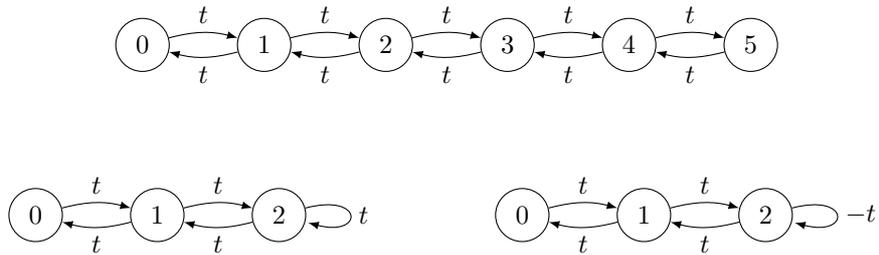

We denote by \smash{$F^+_k$} and \smash{$F^-_k$} the determinants of the
matrices \smash{$1 - A^+_{k-1}$} and \smash{$1 - A^-_{k-1}$}, respectively. We
also denote by \smash{$F^+_{k,\ell}$} and \smash{$F^-_{k,\ell}$} the
$(\ell,0)$ cofactors of the matrices \smash{$1 - A^+_{k}$} and
\smash{$1 - A^-_{k}$}, respectively.

\begin{theorem}\label{quotient}
The polynomial $F_k$ satisfies for all integers~$k$:
\[F_k = F_k^+ F_k^-.\]
Moreover, the generating functions $E_{k,\ell}$ satisfy, for all $\ell$
such that $0\le\ell < k/2$:
\begin{align*}
E_{k,\ell} + E_{k,k-\ell} &= \frac{F_{k,\ell}^+}{F_{k+1}^+}\text;&
E_{k,\ell} - E_{k,k-\ell} &= \frac{F_{k,\ell}^-}{F_{k+1}^-}\text.
\end{align*}
Finally, if $k$ is even and $\ell = k/2$, we have:
\[E_{k,\ell} = \frac{F_{k,\ell}^+}{F_{k+1}^+}.\]
\end{theorem}

\begin{proof}
Let $\mathcal B = (b_0,\dotsc,b_k)$ be the canonical basis of the underlying
vector space of the matrix $A_k$. We denote by $\mathcal B_0$, $\mathcal
B_1$ and $\mathcal B_2$ the following collections of vectors:
\begin{align*}
\mathcal B_0 &= (b_i)_{0\le i<k/2};\\
\mathcal B_1 &= (b_{k/2})\text{ if $k$ is even and }
\varnothing\text{ otherwise;}\\
\mathcal B_2 &= (b_{k-i})_{0\le i<k/2}.
\end{align*}
Since the entry $(i,j)$ of the matrix $A_k$ is equal to the entry
$(k-i,k-j)$, the matrix $A_k$ written as a block matrix with respect to the
basis $(\mathcal B_0,\mathcal B_1, \mathcal B_2)$ looks like:
\[A_k = \left(
\begin{mymatrix}{c|c|c}
B&U&C\arraynl
V&x&V\arraynl
C&U&B
\end{mymatrix}
\right).\]
Let $P$ be the following passage matrix:
\[P = \left(
\begin{mymatrix}{c|c|c}
1&0&1\arraynl
0&1&0\arraynl
1&0&-1
\end{mymatrix}
\right),\]
where $1$ denotes the identity matrix of the appropriate dimension. We
compute:
\[P^{-1}A_kP = \left(
\begin{mymatrix}{c|c|c}
B+C&U&0\arraynl
2V&x&0\arraynl
0&0&B-C
\end{mymatrix}
\right) =
\left(
\begin{mymatrix}{c|c}
A_k^+&0\arraynl
0&A_k^-
\end{mymatrix}
\right).\]
All identities of the theorem are readily derived from this form.
\end{proof}


This result shows that the generating functions $E_{k,\ell}$ can be computed
using smaller polynomials than expected, since the matrices \smash{$A_k^+$}
and \smash{$A_k^+$} (and thus their determinants and cofactors) are about
twice as small as $A_k$. Moreover, we can derive from the proposition a
simplified expression for the generating function of meanders $M_k$. By
writing $M_k$ as the sum of $E_{k,\ell}$ for all $0\le\ell\le k$ and grouping
the terms by pairs, we find:
\begin{equation}\label{Mk-sym}
M_k = \frac{F^+_k(1)}{F^+_{k+1}}.
\end{equation}
This expression involves smaller polynomials that \eqref{Mk} (a simplification
occurs by a factor of \smash{$F^-_{k+1}$}).

\begin{theorem}\label{Fkl-sym}
The generating functions of the polynomials \smash{$F_k^+$} and
\smash{$F_k^-$} are of the form:
\begin{align*}
\sum_{k\ge0} F_k^+z^{k} &= \frac{N^{+}(z)}{D(z^2)}\text;&
\sum_{k\ge0} F_k^-z^{k} &= \frac{N^{-}(z)}{D(z^2)}\text,
\end{align*}
where $D(z)$ is the polynomial defined in Theorem~\ref{Fk} and $N^+(z)$ and
$N^-(z)$ are polynomials. Moreover, the bivariate generating functions of the
polynomials \smash{$F_{k,\ell}^+$} and \smash{$F_{k,\ell}^-$} are of the form:
\begin{align*}
\sum_{k\ge\ell\ge 0} F_{k,\ell}^+ u^\ell z^k
 &= \frac{\widetilde N^{+}(u,z)}{\widetilde D(uz^2)D(z^2)}\text;&
\sum_{k\ge\ell\ge 0} F_{k,\ell}^- u^\ell z^k
 &= \frac{\widetilde N^{-}(u,z)}{\widetilde D(uz^2)D(z^2)}\text,
\end{align*}
where \smash{$\widetilde D(z)$} is the polynomial defined in
Theorem~\ref{Fkl} and \smash{$\widetilde N^+(u,z)$} and \smash{$\widetilde
N^-(u,z)$} are polynomials.
\end{theorem}

\begin{proof}
This theorem is a consequence of a fact hinted at in
Figure~\ref{fig-quotients}: the matrices \smash{$A_k^+$} and \smash{$A_k^-$}
are nearly identical to the matrices $A_{k^+}$ and $A_{k^-}$, respectively,
where \smash{$k^+ = \bigl\lfloor\frac k2\bigr\rfloor$} and
\smash{$k^- = \bigl\lfloor\frac{k-1}2\bigr\rfloor$}. More precisely, the term
$\omega_{k-j-i}$ is zero if $k-j-i > a$; since $j\le k/2$, this is true
whenever $i < k/2 - a$.

For simplicity, we only prove the results associated to the matrix
\smash{$A_k^-$}, but the case of \smash{$A_k^+$} is identical. The proof
follows the same techniques used in the previous sections. If $I$ is a
$a$-subset of $\ldb-a,a-1\rdb$, we denote by \smash{$\mathbf F^-_{k+1}$} the
vector whose $I$-coefficient is:
\[\mathbf F^-_{k}[I] =
\sum_{\sigma\in\mathfrak S^{(I)}_{k^-}} \epsilon(\sigma)
\prod_{i=0}^{k^--1}\bigl(\beta_{\sigma(i)-i}-\omega_{k-1-\sigma(i)-i}\bigr).\]
Assume now that $k-1 > 2a$ and examine the first term of the product: the
above remark entails that $\omega_{k-\sigma(i)-i} = 0$, which means that the
first term is equal to $\beta_{\sigma(i)-i}$. This allows us to repeat the
proof of Proposition~\ref{FkT}. As the matrix \smash{$A^-_k$} minus its first
row and first column is equal to \smash{$A^-_{k-2}$}, we find:
\[\mathbf F^-_k = \mathbf T \mathbf F^-_{k-2}.\]

In the same way, we define the vectors \smash{$\mathbf F^-_{k,\ell}$}; by
repeating the proof of Proposotion~\ref{FklT}, we find, if $k$ is sufficiently
large:
\begin{align*}
\mathbf F^-_{k+2,\ell+2} &= \widetilde{\mathbf T}\mathbf F^-_{k,\ell};\\
\mathbf F^-_{k,0} &= \mathbf U\mathbf F^-_k.
\end{align*}
All the identities of the theorem are derived from these recurrence relations.
\end{proof}

As an example, we take the Fibonacci polynomials, corresponding to the set $S
= \{\pm1\}$ (see Section~\ref{excursions}). The values of the polynomials
\smash{$F_k^+$} and \smash{$F_k^-$} are given by:
\begin{align*}
F_{2k}^+ &= F_k - tF_{k-1};&
F_{2k+1}^+ &= F_{k+1} - t^2F_{k-1};\\
F_{2k}^- &= F_k + tF_{k-1};&
F_{2k+1}^- &= F_k.
\end{align*}
Obviously, these four sequences of polynomials follow the same recurrence
relation as the polynomials $F_k$.

\bibliographystyle{plain}
\bibliography{biblio}{}

\end{document}